%% file: main.tex
\documentclass[11pt]{amsart}

\usepackage{hyperref} 
\usepackage{graphicx}
\usepackage[T1]{fontenc}
\usepackage{mathptmx}
\usepackage{color}
\usepackage{soul}
\usepackage[titletoc]{appendix}
\usepackage{amsmath, amssymb, amsthm}
\usepackage{subcaption}
\usepackage{float}

\DeclareMathOperator{\trace}{trace}
\DeclareMathOperator{\Ad}{Ad}

\DeclareMathOperator{\sech}{sech}

\newcommand{\bbar}{\begin{pmatrix}}
\newcommand{\ebar}{\end{pmatrix}}

\newcommand{\bdm}{\begin{displaymath}}
\newcommand{\edm}{\end{displaymath}}
\newcommand{\beq}{\begin{equation}}
\newcommand{\beqa}{\begin{eqnarray}}
\newcommand{\beqas}{\begin{eqnarray*}}
\newcommand{\eeq}{\end{equation}}
\newcommand{\eeqa}{\end{eqnarray}}
\newcommand{\eeqas}{\end{eqnarray*}}
\newcommand{\dd}{\textup{d}}

\newcommand{\real}{{\mathbb R}}

\newcommand{\glR}{\mathfrak{gl}(2,\real)}
\newcommand{\SLR}{\mathrm{SL}(2,\real)}
\newcommand{\slR}{\mathfrak{sl}(2,\real)}

\newcommand{\pp}{\mathfrak{p}}

\newcommand{\ads}{\mathbb{H}_1^3}
\newcommand{\sff}{\mathrm{I\!I}}
\newcommand{\fff}{\mathrm{I}}

   \newtheorem{theorem}{Theorem}[section]
   \newtheorem{proposition}[theorem]{Proposition}
	  
   \newtheorem{corollary}[theorem]{Corollary}
   \newtheorem{lemma}[theorem]{Lemma}
   \newtheorem{definition}[theorem]{Definition}

 \theoremstyle{remark}
   \newtheorem{example}[theorem]{Example}
   \newtheorem{remark}[theorem]{Remark}

\numberwithin{equation}{section}

\title{Lorentz harmonic maps into the hyperbolic plane and spacelike surfaces in anti-de Sitter 3-space}

\author{Jorge Bravo-Gadea}

\begin{document}

\begin{abstract}
We study the relationship between Lorentz harmonic maps into the hyperbolic plane 
and spacelike surfaces in anti-de Sitter 3-space. Using loop group techniques, 
we develop a DPW-type representation for Lorentz harmonic maps and provide an 
explicit solution of the associated Cauchy problem in terms of a pair of potentials.

We then establish a correspondence between Lorentz harmonic maps and spacelike 
immersions in anti-de Sitter space, identifying conditions under which a harmonic 
map arises as the Gauss map of a surface. In the nondegenerate case, this leads to 
a one-parameter family of spacelike surfaces of constant Gauss curvature, together 
with explicit reconstruction formulas. We also analyze the degenerate case, where 
the Gauss map fails to be an immersion, and show that additional data are required 
to recover the surface.

Finally, we formulate and solve the geometric Cauchy problem for spacelike surfaces 
of constant curvature in anti-de Sitter space, providing a constructive method to 
recover surfaces from prescribed initial data. This establishes a direct link between 
the analytic theory of Lorentz harmonic maps and the geometry of surfaces in Lorentzian 
space forms.
\end{abstract}



\maketitle
\input{introduction2}
\input{preliminaries}
\input{constructionv2}
\input{cauchy_problem}
\input{surfaces}
\input{parallel_surfaces}

\subsection*{Acknowledgments}
This work was supported by the Independent Research Fund Denmark (DFF), 
grant number 9040-00196B.

\bibliographystyle{plain}
\bibliography{References}

\end{document}

%% file: introduction2.tex
\section{Introduction}

The study of surfaces in three-dimensional homogeneous spaces has long been a central topic 
in differential geometry, with deep connections to integrable systems and harmonic map theory. 
A classical theme is the relationship between geometric properties of an immersion—such as 
constant mean curvature (CMC) or constant Gauss curvature (CGC)—and the analytic behaviour 
of its Gauss map (see, e.g., \cite{MR3726907,MR831039}).

In the Riemannian setting, this interplay is well understood. In Euclidean space, constant mean 
curvature surfaces are characterized by the harmonicity of their Gauss map \cite{ruh1970tension}, 
while surfaces of constant negative curvature are governed by integrable equations such as the 
sine--Gordon equation \cite{bobenko1996discrete}. In hyperbolic space, Bryant showed that surfaces 
of mean curvature one admit a holomorphic representation formula \cite{bryant1987surfaces}, 
revealing a deep connection between surface theory and complex analysis.

These phenomena are part of a broader framework relating differential geometry and integrable 
systems. In particular, harmonic maps into symmetric spaces can be described via loop group 
factorizations \cite{MR1664887,MR900587,MR1630443,MR3087328}, providing a unified approach to the 
construction of surfaces. This perspective has been further developed through the DPW method 
and its generalizations \cite{MR1664887,MR2491604}, which allow one to construct surfaces from 
holomorphic data.

A key feature of this theory is the existence of correspondences between surfaces in different 
space forms. Classical results such as the Lawson correspondence relate minimal surfaces in 
Euclidean space with constant mean curvature surfaces in spheres \cite{lawson1970complete}, 
while transformation theories such as Bäcklund and Darboux transformations generate new 
solutions from given ones and play a central role in integrable surface theory \cite{tenenblat1998transformations,MR581396}. 
These ideas extend to Lorentzian geometry, where surfaces in Minkowski, de Sitter, and anti-de 
Sitter spaces are related through analogous correspondences.

In the Lorentzian setting, spacelike and timelike surfaces with prescribed curvature conditions 
have been extensively studied. In particular, surfaces with Lorentz harmonic Gauss maps provide 
a natural generalization of the classical theory, and their geometric Cauchy problem has been 
analyzed using integrable systems techniques \cite{MR3019511,MR3223570,MR2158167}. These results highlight 
the deep interaction between harmonic map theory and the geometry of surfaces in Lorentzian 
space forms.

In recent years, attention has turned to the study of surfaces in anti-de Sitter 3-space 
$\ads$, which can be realized as the Lie group $\mathrm{SL}(2,\mathbb{R})$ endowed with a 
bi-invariant Lorentzian metric \cite{MR425012}. This Lie group structure makes $\ads$ particularly 
suitable for the application of loop group methods and has led to the development of representation 
formulae and classification results for various classes of surfaces \cite{MR2251429,MR1282222,MR2833568}. 
Further developments have explored the role of the Gauss map and the geometry of surfaces in Lie 
groups more generally \cite{MR3537970,MR3468639}.

In this paper, we investigate the correspondence between Lorentz harmonic maps and spacelike 
surfaces in $\ads$. Our approach is based on a DPW-type construction \cite{MR1664887,MR2491604}, which allows 
us to solve the Cauchy problem for Lorentz harmonic maps into the hyperbolic plane $H^2$ in a 
constructive way. While existence and uniqueness follow from general analytic results such as 
the Cauchy--Kowalevski theorem \cite{MR3012036}, our method provides explicit representation 
formulae that are essential for geometric applications. Similar constructions for constant mean curvature surfaces via loop group methods 
can be found in \cite{https://doi.org/10.48550/arxiv.math/0602570}.

We then apply this construction to the geometry of spacelike surfaces in $\ads$. We show that, 
under suitable conditions, Lorentz harmonic maps arise as Gauss maps of spacelike immersions, 
and we obtain a representation formula for such surfaces. In particular, we establish a 
correspondence between harmonic Gauss maps and spacelike surfaces of constant Gauss curvature, 
providing a unified framework for their construction.

A key feature of our approach is that it allows us to solve the geometric Cauchy problem for 
surfaces of constant curvature in a systematic way. Given appropriate initial data along a curve, 
we construct a unique spacelike immersion realizing these data, thus linking the analytic theory 
of harmonic maps with the geometric theory of surfaces.

Finally, we study the parallel surface construction and its effect on curvature. We show that 
the harmonic Gauss map naturally generates families of associated surfaces, relating constant 
Gauss curvature and constant mean curvature geometries in $\ads$. This can be interpreted as a 
Lorentzian analogue of classical transformation theories for surfaces \cite{tenenblat1998transformations,MR3087328}, 
and allows us to reduce the geometric Cauchy problem for constant mean curvature surfaces to the 
corresponding problem for constant Gauss curvature surfaces.

The paper is organized as follows. In Section 2 we recall the necessary background on $\ads$, 
harmonic maps, and frame theory. In Section 3 we develop the loop group formulation and the 
DPW construction. Section 4 is devoted to the solution of the Cauchy problem. In Section 5 
we study the correspondence with spacelike surfaces and derive the curvature conditions. 
Section 6 addresses the geometric Cauchy problem, and Section 7 discusses parallel surfaces 
and associated families.

%% file: preliminaries.tex
\section{Preliminaries}

\subsection{The hyperbolic plane in $\mathfrak{sl}(2,\mathbb{R})$}

In this section we recall the matrix models for the hyperbolic plane and anti-de Sitter 3-space that will be used throughout the paper. These models are particularly well suited for the application of loop group methods (see, e.g., \cite{MR0719023,MR3468639}).

Let $\glR$ be the space of real $2 \times 2$ matrices. Consider the basis
\[
e_0 = \begin{pmatrix} 1 & 0 \\ 0 & 1 \end{pmatrix}, \quad
e_1 = \begin{pmatrix} 0 & -1 \\ 1 & 0 \end{pmatrix}, \quad
e_2 = \begin{pmatrix} 0 & 1 \\ 1 & 0 \end{pmatrix}, \quad
e_3 = \begin{pmatrix} -1 & 0 \\ 0 & 1 \end{pmatrix}.
\]

We endow $\glR$ with the bilinear form
\[
\langle X, Y \rangle = -\tfrac{1}{2}\operatorname{trace}(X \overline{Y}),
\]
where $\overline{Y}$ denotes the adjoint transpose of $Y$ \cite{MR425012}. With respect to this metric, $\glR$ becomes isometric to $\mathbb{R}^{2,2}$, and one verifies that
\[
\langle X, X \rangle = -\det X.
\]

In particular, the special linear group can be identified as
\[
\mathrm{SL}(2,\mathbb{R}) = \{ X \in \glR \mid \det X = 1 \}
= \{ X \in \glR \mid \langle X, X \rangle = -1 \},
\]
which realizes anti-de Sitter 3-space as a quadric in $\mathbb{R}^{2,2}$ \cite{MR2251429,MR1282222}.

Restricting the metric to the Lie algebra
\[
\slR = \mathrm{Span}\{e_1, e_2, e_3\},
\]
we obtain a Lorentzian inner product
\[
\langle X, Y \rangle = \tfrac{1}{2}\operatorname{trace}(X \bar{Y}),
\]
with respect to which $\slR$ is isometric to Minkowski 3-space $\mathbb{R}^{2,1}$ \cite{MR0719023}.

The hyperbolic plane is then realized as the quadric
\[
H^2 = \{ X \in \slR \mid \langle X, X \rangle = -1 \}.
\]

\subsection{The harmonic equation}

We now recall the notion of Lorentz harmonic maps into the hyperbolic plane in the present setting.

Let $\Omega \subset \mathbb{R}^{1,1}$ be a domain endowed with the standard Lorentzian metric, and let
\[
\Sigma_\sigma = \{ X \in \mathbb{R}^{2,1} \mid \langle X, X \rangle = \sigma \}, \quad \sigma = \pm 1,
\]
be the pseudo-sphere with the induced metric.

Consider a smooth map
\[
N : \Omega \longrightarrow \Sigma_\sigma.
\]
We say that $N$ is \emph{Lorentz harmonic} if it is a critical point of the energy functional associated to the Lorentzian metric on $\Omega$ \cite{MR3726907}. In coordinates, this condition can be expressed as a wave-type equation (see, e.g., \cite{MR3012036}).

Let $(x,y)$ be standard Lorentzian coordinates on $\Omega$. Then $N$ is harmonic if and only if
\[
N_{xx} - N_{yy}
\]
is pointwise parallel to $N$.

For our purposes, it is more convenient to work in null coordinates
\[
s = \tfrac{1}{2}(x+y), \qquad t = \tfrac{1}{2}(x-y).
\]
In these coordinates, the harmonicity condition becomes
\[
N_{xx} - N_{yy} = 4 N_{st},
\]
so that $N$ is harmonic if and only if $N_{st}$ is parallel to $N$.

Equivalently, using the Lie algebra structure, this condition can be written as
\[
[N, N_{st}] = 0,
\]
as appears, for instance, in \cite{MR1630443}.

\subsection{Frames for maps into $H^2$}

We now introduce a frame formulation for maps into the hyperbolic plane, which will be fundamental for the loop group construction \cite{MR1664887,MR900587} developed later.

Let $\Omega \subset \mathbb{R}^{1,1}$ be a simply connected domain, and let
\[
\nu : \Omega \to H^2 \subset \slR
\]
be a smooth map.

\begin{definition}
An $\mathrm{SL}(2,\mathbb{R})$-frame for $\nu$ is a smooth map
\[
F : \Omega \to \mathrm{SL}(2,\mathbb{R})
\]
such that
\[
\nu = \operatorname{Ad}_F e_1,
\]
see, for example, \cite{MR1630443}.
\end{definition}

The Maurer--Cartan form of the frame $F$ is defined by
\[
\alpha = F^{-1} dF.
\]
This is a $\slR$-valued $1$-form on $\Omega$, see \cite{MR900587}.

We decompose $\slR$ as
\[
\slR = \mathfrak{k} \oplus \mathfrak{p}, \qquad 
\mathfrak{k} = \mathrm{Span}\{e_1\}, \quad 
\mathfrak{p} = \mathrm{Span}\{e_2, e_3\}.
\]
as is standard in the theory of symmetric spaces (see, e.g., \cite{MR1630443}). Accordingly, we write
\[
\alpha = \alpha_{\mathfrak{k}} + \alpha_{\mathfrak{p}}.
\]

In null coordinates $(x,y)$ on $\Omega$, we can write
\[
\alpha = (U_{\mathfrak{k}} + U_{\mathfrak{p}})\,dx + (V_{\mathfrak{k}} + V_{\mathfrak{p}})\,dy,
\]
where $U_{\mathfrak{k}}, V_{\mathfrak{k}} \in \mathfrak{k}$ and $U_{\mathfrak{p}}, V_{\mathfrak{p}} \in \mathfrak{p}$.

\medskip

In order to simplify the equations, we choose a special frame adapted to $\nu$. After a suitable reparametrization of the null coordinates, we may assume that
\[
\langle \nu_x, \nu_x \rangle = 1.
\]
Then we can construct a frame $F$ such that
\[
\nu = \operatorname{Ad}_F e_1, \qquad 
\nu_x = \operatorname{Ad}_F e_3, \qquad 
\nu_y = \operatorname{Ad}_F (a e_2 + b e_3),
\]
for some smooth functions $a,b : \Omega \to \mathbb{R}$.

With respect to this frame, the Maurer--Cartan form takes the form
\[
\alpha = (c e_1 + A e_2 + B e_3)\,dx + (d e_1 + A' e_2 + B' e_3)\,dy.
\]

A direct computation shows that
\[
U_{\mathfrak{p}} = -\tfrac{1}{2} e_2, \qquad U_{\mathfrak{k}} = c e_1,
\]
and
\[
V_{\mathfrak{p}} = -\tfrac{b}{2} e_2 + \tfrac{a}{2} e_3, \qquad V_{\mathfrak{k}} = d e_1.
\]

Finally, the Maurer--Cartan equation
\[
d\alpha + \alpha \wedge \alpha = 0
\]
(see, for example, \cite{MR900587}) is equivalent to the system
\beqa\label{MC-system}
\begin{cases}
\partial_y U_{\mathfrak{k}} - \partial_x V_{\mathfrak{k}} = [U_{\mathfrak{p}}, V_{\mathfrak{p}}], \\
\partial_y U_{\mathfrak{p}} - \partial_x V_{\mathfrak{p}} = [U_{\mathfrak{k}}, V_{\mathfrak{p}}] + [U_{\mathfrak{p}}, V_{\mathfrak{k}}].
\end{cases}
\eeqa

%% file: constructionv2.tex
\section{Loop group construction}

In this section we develop a loop group formulation of Lorentz harmonic maps into $H^2$, which leads to a constructive representation via a DPW-type method (see, e.g., \cite{MR1664887,MR2491604}).
\subsection{Extended frames}

Let $\nu : \Omega \to H^2$ be a smooth map and let $F : \Omega \to \mathrm{SL}(2,\mathbb{R})$ be a frame as in Section 2.3, with Maurer--Cartan form
\[
\alpha = F^{-1} dF = (U_{\mathfrak{k}} + U_{\mathfrak{p}})\,dx + (V_{\mathfrak{k}} + V_{\mathfrak{p}})\,dy.
\]

For $\lambda \in \mathbb{R}^\times$, define the $\lambda$-dependent $1$-form
\[
\hat{\alpha} = (U_{\mathfrak{k}} + \lambda U_{\mathfrak{p}})\,dx + (V_{\mathfrak{k}} + \lambda^{-1} V_{\mathfrak{p}})\,dy.
\]

This construction is standard in the loop group approach to harmonic maps (see, e.g., \cite{MR900587,MR1630443}). The following result is fundamental.

\begin{lemma}
The map $\nu$ is harmonic if and only if the $1$-form $\hat{\alpha}$ satisfies the Maurer--Cartan equation
\[
d\hat{\alpha} + \hat{\alpha} \wedge \hat{\alpha} = 0
\]
for all $\lambda \in \mathbb{R}^\times$. This characterization is well known in the integrable systems approach to harmonic maps (see, e.g., \cite{MR1664887,MR1630443}).
\end{lemma}
\begin{proof}
Consider the Maurer--Cartan equation for the loop algebra-valued form
\[
\hat{\alpha} = (U_k + \lambda U_p)\,dx + (V_k + \lambda^{-1} V_p)\,dy.
\]

A direct computation yields
\[
d\hat{\alpha} + \hat{\alpha} \wedge \hat{\alpha}
=
\Big(
\partial_y U_k - \partial_x V_k + [U_p,V_p]
\Big)\,dx \wedge dy
\]
\[
+\, \lambda\Big(
\partial_y U_p - [U_p,V_k]
\Big)\,dx \wedge dy
+\, \lambda^{-1}\Big(
-\partial_x V_p + [U_k,V_p]
\Big)\,dx \wedge dy.
\]

Requiring that this expression vanishes for all \(\lambda\) is equivalent to the system
\[
\partial_y U_k - \partial_x V_k = [U_p,V_p],
\]
\[
\partial_y U_p = [U_p,V_k], \qquad
\partial_x V_p = [U_k,V_p].
\]

This system is the integrability condition for the extended frame \(\hat F\), and is therefore stronger than the Maurer--Cartan equation for \(F\). We will refer to this system as the \emph{split Maurer--Cartan system}.
\end{proof}
As a consequence of the previous lemma, for a harmonic map $\nu$ we can integrate the equation
\[
\hat{\alpha} = \hat{F}^{-1} d\hat{F},
\]
to obtain a map
\[
\hat{F} : \Omega \to \Lambda \mathrm{SL}(2,\mathbb{C}),
\]
called an \emph{extended frame}. Such frames play a central role in the loop group formulation of harmonic maps \cite{MR900587}.

By construction, $\hat{F}$ depends on the parameter $\lambda$, and satisfies
\[
\hat{F}|_{\lambda=1} = F.
\]
In particular, the original map $\nu$ can be recovered as
\[
\nu = \operatorname{Ad}_{\hat{F}|_{\lambda=1}} e_1.
\]

\subsection{The d'Alembert construction}

We now describe explicitly how to construct an extended frame from a pair of potentials. This construction is a Lorentzian analogue of the classical d'Alembert method and is closely related to the integrable systems approach to harmonic maps (see, e.g., \cite{MR581396,MR2158167}).

Let $(\mathcal{X},\mathcal{Y})$ be a potential pair as introduced in the DPW framework (see \cite{MR1664887,MR2491604}), that is, two $\Lambda \sl(2,\mathbb{C})$-valued $1$-forms depending only on $x$ and $y$, respectively, of the form
\[
\mathcal{X}(x) = \sum_{k \le 1} A_k(x)\lambda^k\,dx, \qquad 
\mathcal{Y}(y) = \sum_{k \ge -1} B_k(y)\lambda^{-k}\,dy.
\]

We integrate the differential equations
\[
X^{-1} dX = \mathcal{X}, \qquad 
Y^{-1} dY = \mathcal{Y},
\]
with initial conditions $X(x_0) = I$, $Y(y_0) = I$.
This yields maps
\[
X : I_x \to \Lambda \mathrm{SL}(2,\mathbb{C}), \qquad
Y : I_y \to \Lambda \mathrm{SL}(2,\mathbb{C}).
\]

Next, we define
\[
\Phi(x,y) = X(x)^{-1} Y(y).
\]

Assuming that $\Phi(x,y)$ lies in the big cell, which holds on an open dense subset, we perform a Birkhoff decomposition (see, e.g., \cite{MR900587})
\[
\Phi(x,y) = H_-(x,y)\, H_+(x,y),
\]
where $H_-$ extends holomorphically to $|\lambda| > 1$ and $H_+$ to $|\lambda| < 1$.

The extended frame is then defined by
\[
\hat{F}(x,y) = Y(y)\, H_+(x,y)^{-1}
= X(x)\, H_-(x,y),
\]
following the standard DPW construction \cite{MR1664887}.

By construction, $\hat{F}$ satisfies
\[
\hat{F}^{-1} d\hat{F}
= (U_{\mathfrak{k}} + \lambda U_{\mathfrak{p}})\,dx
+ (V_{\mathfrak{k}} + \lambda^{-1} V_{\mathfrak{p}})\,dy,
\]
and therefore defines a Lorentz harmonic map into $H^2$. This procedure provides an explicit method for constructing solutions to the harmonic map equation from holomorphic data, which is one of the key features of the DPW approach.

\subsection{A first explicit example}

We now present a completely explicit example illustrating the construction. The goal is to show how the method works in practice, avoiding any formal or implicit steps, in the spirit of explicit constructions in integrable surface theory.

A natural first choice is to consider the case where one of the potentials vanishes. 
Geometrically, this corresponds to maps depending on a single variable, and analytically it eliminates the need for a nontrivial Birkhoff decomposition. 
This allows us to carry out all computations explicitly.

We therefore consider the potential pair
\[
\mathcal{X}(x) = \lambda A(x)\,dx, \qquad
\mathcal{Y}(y) = 0,
\]
where
\[
A(x) = \frac{1}{2}(\tanh x\, e_3 - \operatorname{sech} x\, e_2).
\]

We integrate the system
\[
X^{-1} dX = \mathcal{X}, \qquad Y^{-1} dY = \mathcal{Y},
\]
with initial conditions $X(0)=I$, $Y(0)=I$.

Since $\mathcal{Y}=0$, we obtain immediately
\[
Y(y) = I.
\]

To compute $X(x)$, observe that the equation
\[
X^{-1} dX = \lambda A(x)\,dx
\]
reduces, for $\lambda=1$, to
\[
X^{-1} X_x = A(x).
\]
A direct integration yields
\[
X(x) =
\begin{pmatrix}
\sqrt{\cosh x} & -\dfrac{\sinh x}{\sqrt{\cosh x}} \\
0 & \dfrac{1}{\sqrt{\cosh x}}
\end{pmatrix}.
\]

We now construct
\[
\Phi(x,y) = X(x)^{-1} Y(y) = X(x)^{-1}.
\]

Since $\Phi$ depends only on $x$ and contains no negative powers of $\lambda$, 
its Birkhoff decomposition is trivial:
\[
\Phi = H_- H_+, \qquad H_- = I, \quad H_+ = \Phi.
\]

Thus the extended frame is given by
\[
\hat{F}(x,y) = X(x).
\]

Evaluating at $\lambda=1$, we obtain the frame
\[
F(x,y) = X(x).
\]

Finally, the associated harmonic map is
\[
\nu = \operatorname{Ad}_F e_1 = F e_1 F^{-1}.
\]

A direct computation gives
\[
F e_1 =
\begin{pmatrix}
\sqrt{\cosh x} & -\dfrac{\sinh x}{\sqrt{\cosh x}} \\
0 & \dfrac{1}{\sqrt{\cosh x}}
\end{pmatrix}
\begin{pmatrix}
0 & -1 \\
1 & 0
\end{pmatrix}
=
\begin{pmatrix}
-\dfrac{\sinh x}{\sqrt{\cosh x}} & -\sqrt{\cosh x} \\
\dfrac{1}{\sqrt{\cosh x}} & 0
\end{pmatrix},
\]

and
\[
F^{-1} =
\begin{pmatrix}
\dfrac{1}{\sqrt{\cosh x}} & \dfrac{\sinh x}{\sqrt{\cosh x}} \\
0 & \sqrt{\cosh x}
\end{pmatrix}.
\]

Multiplying, we obtain
\[
\nu(x,y) = F e_1 F^{-1}
=
\begin{pmatrix}
-\sinh x & -\cosh x \\
\cosh x & \sinh x
\end{pmatrix}.
\]
\begin{remark}
This example illustrates the general fact that any map $\nu$ depending only on one of the null coordinates is trivially Lorentz harmonic, since the mixed derivative $\nu_{st}$ vanishes identically, and therefore $[\nu,\nu_{st}] = 0$.
\end{remark}

%% file: cauchy_problem.tex
\section{The Cauchy problem}

In this section we solve the Cauchy problem for Lorentz harmonic maps into $H^2$ using the loop group method developed in Section~3.

More precisely, given suitable initial data along a characteristic curve, we construct a harmonic map realizing these data. While existence and uniqueness follow from general analytic results such as the Cauchy--Kowalevski theorem (see, e.g., \cite{MR3012036}), our goal is to obtain an explicit and constructive solution using integrable systems techniques.
\medskip

Let $I \subset \mathbb{R}$ be an open interval. We consider analytic initial data
\[
N_0 : I \to H^2, \qquad N_1 : I \to S^2_1,
\]
satisfying
\[
\langle N_1(t), N_0(t) \rangle = 0,
\]
and the nondegeneracy condition
\[
\langle N_1(t), N_0'(t) - N_1(t) \rangle \not\equiv 0.
\]

Geometrically, $N_0$ prescribes the value of the map along the curve $x=y$, while $N_1$ determines its derivative in the $x$-direction.

\begin{theorem}
Let $I \subset \mathbb{R}$ be an open interval. Given analytic Cauchy data $N_0$ and $N_1$ as above, there exists a Lorentz harmonic map
\[
\nu : I \times I \to H^2,
\]
such that
\[
\nu(t,t) = N_0(t), \qquad \nu_x(t,t) = N_1(t).
\]
\end{theorem}
This result can be viewed as a constructive version of the geometric Cauchy problem for harmonic maps, in the spirit of integrable systems approaches (see, e.g., \cite{MR2158167,MR3019511}).
\begin{proof}
We seek a harmonic map $\nu : I \times I \to H^2$ satisfying the prescribed initial conditions along the diagonal $x=y$, following the loop group strategy developed in Section~3.

Introduce coordinates
\[
t = \frac{x+y}{2}, \qquad s = \frac{x-y}{2},
\]
so that the diagonal corresponds to $s=0$. The initial conditions become
\[
\nu(t,0) = N_0(t), \qquad \nu_x(t,0) = N_1(t).
\]

In particular, we have
\[
\partial_t \nu(t,0) = N_0'(t).
\]

From the relation $\partial_t = \partial_x + \partial_y$, it follows that along the diagonal
\[
\nu_y = \partial_t \nu - \partial_x \nu = N_0'(t) - N_1(t).
\]

We now compute the coefficients of the Maurer--Cartan form, as in the frame formulation of harmonic maps described in Section~2, along the diagonal. 
Recall that
\[
a = \frac{1}{2} \langle \nu_y, [\nu_x,\nu] \rangle, 
\qquad 
b = \langle \nu_x, \nu_y \rangle.
\]

Substituting the expressions above, we obtain
\[
a = \frac{1}{2} \langle N_0'(t) - N_1(t), [N_1(t), N_0(t)] \rangle
= \langle N_0'(t), [N_1(t), N_0(t)] \rangle,
\]
and
\[
b = \langle N_1(t), N_0'(t) - N_1(t) \rangle.
\]

\medskip

We now compute the derivative of $a$. The key observation is that, by harmonicity, the mixed derivative $\nu_{xy}$ is orthogonal to both $\nu$ and $\nu_x$, which simplifies the computation.

A direct calculation yields
\[
a_x = \langle N_0'(t) - N_1(t), [N_1'(t), N_0(t)] \rangle.
\]

Thus
\[
c = \frac{a_x}{2b}
= \frac{\langle N_0'(t) - N_1(t), [N_1'(t), N_0(t)] \rangle}
{2\langle N_1(t), N_0'(t) - N_1(t) \rangle}.
\]

\medskip

We now encode this information in a $\lambda$-dependent Maurer--Cartan form along the initial curve. This step corresponds to the introduction of a normalized potential in the DPW method (see \cite{MR1664887,MR2491604}).
\[
\hat{\alpha}(u)
=
\left(
c e_1 - \frac{1}{2} e_2 \lambda
+ \frac{1}{2}(a e_3 - b e_2)\lambda^{-1}
\right) du.
\]

This form determines a pair of potentials
\[
\mathcal{X}(x) = \hat{\alpha}(x)\,dx, \qquad
\mathcal{Y}(y) = \hat{\alpha}(y)\,dy.
\]

We now apply the d’Alembert construction of Section~3, which provides an explicit integration procedure for the harmonic map equation. Integrating
\[
X^{-1} dX = \mathcal{X}, \qquad
Y^{-1} dY = \mathcal{Y},
\]
with initial condition $X(x_0) = Y(x_0) = K$, where $K \in \mathrm{SL}(2,\mathbb{R})$ satisfies
\[
\operatorname{Ad}_K e_1 = N_0(x_0),
\]
we obtain maps
\[
X(x), \qquad Y(y).
\]

Along the diagonal $x=y$ we have $X(x) = Y(x)$, and therefore the Birkhoff decomposition of (see, e.g., \cite{MR900587})
\[
\Phi(x,y) = X(x)^{-1} Y(y)
\]
is trivial along the diagonal:
\[
\Phi = H_- H_+, \qquad H_- = H_+ = I.
\]

Thus the extended frame is given by
\[
\hat{F}(x,y) = X(x) H_-(x,y),
\]
which coincides with $X(x)$ along the diagonal.

Let
\[
F = \hat{F}|_{\lambda=1}, \qquad
\nu = \operatorname{Ad}_F e_1.
\]

By construction, $\hat{F}$ satisfies the Maurer--Cartan equation, hence $\nu$ is harmonic, which is the standard mechanism in the loop group approach to harmonic maps (see \cite{MR1664887}).

\medskip

Finally, we verify the initial conditions. Along the diagonal,
\[
\hat{F}^{-1} d\hat{F} = \mathcal{X}(x) = \hat{\alpha}(x)\,dx.
\]

Separating the $dx$ and $dy$ contributions, we obtain
\[
\hat{F}^{-1} d\hat{F}
=
\left(c e_1 - \frac{1}{2} e_2 \lambda\right) dx
-
\frac{1}{2}(a e_3 - b e_2)\lambda^{-1} dy.
\]

This implies
\[
\nu_x = \operatorname{Ad}_F [F^{-1}F_x, e_1] = \operatorname{Ad}_F e_3,
\]
and
\[
\nu_y = \operatorname{Ad}_F (a e_2 + b e_3).
\]

In particular,
\[
\nu(t,0) = N_0(t), \qquad \nu_x(t,0) = N_1(t),
\]
which shows that $\nu$ solves the Cauchy problem.
\end{proof}

\begin{example}
We illustrate the construction by solving explicitly the Cauchy problem using the procedure of the proof.

Consider the initial data
\[
N_0(t)
=
\begin{pmatrix}
-\dfrac{2t}{1-t^2} & -\dfrac{1+t^2}{1-t^2} \\
\dfrac{1+t^2}{1-t^2} & \dfrac{2t}{1-t^2}
\end{pmatrix},
\]
and
\[
N_1(t)
=
\frac{1}{(1-t^2)^2}
\begin{pmatrix}
-(1+t^2) & -2t \\
2t & (1+t^2)
\end{pmatrix}.
\]

A direct computation shows that
\[
\langle N_1(t), N_0(t) \rangle = 0,
\]
and that the nondegeneracy condition is satisfied.

\medskip

We now compute the coefficients appearing in the construction. Using the formulas of the proof, we obtain
\[
b = \langle N_1, N_0' - N_1 \rangle = \frac{1}{(1-t^2)^2},
\]
and
\[
a = \langle N_0', [N_1, N_0] \rangle = \frac{2t}{(1-t^2)^2}.
\]

At this point, the coefficient \(c\) depends on the choice of frame. 
More precisely, we may perform a gauge transformation by a map \(G(t)\) taking values in the stabilizer of \(e_1\), given explicitly by
\[
G(t) =
\begin{pmatrix}
\cosh \theta(t) & \sinh \theta(t)\\
\sinh \theta(t) & \cosh \theta(t)
\end{pmatrix}.
\]
Under this transformation, the Maurer--Cartan form changes as
\[
\hat{\alpha} \;\mapsto\; G^{-1} \hat{\alpha} G + G^{-1} dG,
\]
and a direct computation shows that the coefficient \(c\) transforms as
\[
c \;\mapsto\; \tilde c = c + \theta'(t).
\]

Choosing \(\theta\) such that
\[
\theta'(t) = \frac{t}{1-t^2} - c(t),
\]
we may assume without loss of generality that
\[
\tilde c = \frac{t}{1-t^2}.
\]

\medskip

We therefore obtain the $\lambda$-dependent Maurer--Cartan form
\[
\hat{\alpha}(t)
=
\left(
\frac{t}{1-t^2} e_1
- \frac{1}{2} e_2 \lambda
+ \frac{1}{2}
\left(
\frac{2t}{(1-t^2)^2} e_3
- \frac{1}{(1-t^2)^2} e_2
\right)\lambda^{-1}
\right) dt.
\]

This determines the potentials
\[
\mathcal{X}(x) = \hat{\alpha}(x)\,dx, \qquad
\mathcal{Y}(y) = \hat{\alpha}(y)\,dy.
\]

\medskip

Although the extended frame depends on the loop parameter \(\lambda\), in this example we perform the integration at \(\lambda=1\), which is sufficient to recover the harmonic map. We now integrate
\[
X^{-1} dX = \mathcal{X}, \qquad
Y^{-1} dY = \mathcal{Y}.
\]

A direct integration yields
\[
X(x)
=
\begin{pmatrix}
\dfrac{1}{\sqrt{1-x^2}} & \dfrac{x}{\sqrt{1-x^2}} \\
\dfrac{x}{\sqrt{1-x^2}} & \dfrac{1}{\sqrt{1-x^2}}
\end{pmatrix},
\]
and similarly
\[
Y(y)
=
\begin{pmatrix}
\dfrac{1}{\sqrt{1-y^2}} & \dfrac{y}{\sqrt{1-y^2}} \\
\dfrac{y}{\sqrt{1-y^2}} & \dfrac{1}{\sqrt{1-y^2}}
\end{pmatrix}.
\]

\medskip

We form
\[
\Phi(x,y) = X(x)^{-1} Y(y)=\frac{1}{\sqrt{(1-x^2)(1-y^2)}}\begin{pmatrix}
    1-xy & y-x\\y-x& 1-xy
\end{pmatrix},
\]
and perform the Birkhoff decomposition
\[
\Phi = H_- H_+.
\]

In this case, a direct computation shows that the decomposition is given by
\[
H_- =
\begin{pmatrix}
1 & 0 \\
\frac{y-x}{1-xy} & 1
\end{pmatrix}, \qquad
H_+ =
\begin{pmatrix}
\frac{1}{1-xy} & \frac{x}{1-xy} \\
0 & 1
\end{pmatrix}.
\]

\medskip

The extended frame is then
\[
\hat{F} = X(x) H_-,
\]
and evaluating at $\lambda=1$ we obtain
\[
F(x,y)
=
\begin{pmatrix}
\dfrac{1}{1-xy} & \dfrac{x}{1-xy} \\
\dfrac{y}{1-xy} & \dfrac{1}{1-xy}
\end{pmatrix}.
\]

\medskip

Finally, the associated harmonic map is
\[
\nu = \operatorname{Ad}_F e_1,
\]
which yields
\[
\nu(x,y)
=
\begin{pmatrix}
-\dfrac{x+y}{1-xy} & -\dfrac{1+xy}{1-xy} \\
\dfrac{1+xy}{1-xy} & \dfrac{x+y}{1-xy}
\end{pmatrix}.
\]

This provides a fully explicit solution of the Cauchy problem.
\end{example}

%% file: surfaces.tex
The previous sections provide a complete description of Lorentz harmonic maps 
into $H^2$ via loop group methods, together with an explicit solution of the Cauchy problem. 
We now turn to the geometric interpretation of these maps, following a classical philosophy 
in differential geometry relating harmonic maps and surface theory (see, e.g., \cite{ruh1970tension,MR1630443}).

In particular, we investigate when a Lorentz harmonic map can be realized 
as the Gauss map of a spacelike surface in anti-de Sitter space, and how 
such surfaces can be reconstructed from the associated frame.

\section{Spacelike surfaces in $\ads$ with Lorentz harmonic Gauss map with respect to the second fundamental form}
Define anti-de Sitter 3-space as the quadric
\beqas
\ads=\{ x\in\mathbb{R}^4_2\mid \langle x,x\rangle =-1\}
\eeqas
with the induced metric (see, e.g., \cite{MR0719023,MR3468639}). Let $\nu:\Omega\rightarrow\mathbb{H}^2$ be a harmonic map. In this section we find the conditions under which $\nu$ is the Gauss map of an immersion $f:\Omega\rightarrow\ads$ such that the pull-back of $\sff$ endows $\Omega$ with the same Lorentz structure. We then construct a unique CGC spacelike immersion from appropiate data along a curve.

\subsection{Anti-de Sitter 3-space $\ads$ as a Lie group}

Following the construction done in Section 2, we can identify
\beqas
\ads=\{x\in\mathfrak{gl}(2,\mathbb{R})\mid \langle x,x\rangle =-1\}=\{x\in\mathfrak{gl}(2,\mathbb{R})\mid \det x=1\}=\SLR.
\eeqas
This identification equips $\ads$ with a Lie group structure and a bi-invariant Lorentzian metric, 
which plays a fundamental role in the integrable systems approach (see \cite{MR425012}).
With this metric $\slR$ is isometric to the Lorentz-Minkowski 3-space and its Lie bracket satisfies
\beqas
[e_1,e_2]=2e_3,\quad [e_2,e_3]=-2e_1,\quad [e_3,e_1]=2e_2,
\eeqas
which shows that the bracket is twice the cross-product,
\beqas
[X,Y]=2(X\times Y),\quad X,Y\in\slR.
\eeqas
The following identities are standard consequences of the Lie algebra structure 
and will be used repeatedly in the sequel (see, e.g., \cite{MR0719023}).
\begin{proposition}\label{proplb}:
Let $X,Y,Z\in\slR$, and assume $Z$ is orthogonal to both $X$ and $Y$, then:
\begin{enumerate}
    \item $\langle X,[X,Y]\rangle=0$,
    \item $\langle [X,Y],[X,Y]\rangle=4(\langle X,Y\rangle^2-\langle X,X\rangle\langle Y,Y\rangle )$,
    \item $[[Z,X],[Z,Y]]=-4\langle Z,Z\rangle[X,Y]$,
    \item $\langle [Z,X],[Z,Y]\rangle =-4\langle Z,Z\rangle\langle X,Y\rangle$.
\end{enumerate}
    
\end{proposition}

\subsection{Coordinate frame}In the classical correspondence between surfaces and their Gauss maps, 
it is often assumed that the Gauss map is an immersion. However, this condition 
is not strictly necessary in the present setting. Even when the Gauss map fails 
to be immersive at certain points, it still satisfies the harmonicity equation, 
and the geometric construction can be extended to this degenerate case.

Let $\nu:\Omega\rightarrow\mathbb{H}^2$ be a smooth map and assume there is an immersion $f:\Omega\rightarrow\ads$ with $\langle f^{-1}\dd f,\nu\rangle =0$ and such that $\sff$ pull back to the same Lorentz structure on $\Omega$. The immersion $f$ must be spacelike, thus $\det I>0$. The second fundamental form $\sff$ pulls-back to $\Omega$ as a Lorentz metric at points where $\det \sff<0$. This is equivalent to the condition $K+1<0$ on the extrinsic curvature $K$ of $f$.

Let $F$ be a frame for $\nu$, as introduced in Section~2.3 (see also \cite{MR1630443}). And $(x,y)$ a null coordinate system or, equivalently, an asymptotic coordinate system for $f$. There exist two functions $\omega_1,\omega_2:\Omega \rightarrow \slR$ orthogonal to $e_1$ such that
\beqas
f^{-1}f_x=\Ad_F \omega_1,\quad f^{-1}f_y=\Ad_F \omega_2,\quad \nu=\Ad_F e_1.
\eeqas
In these coordinates $\sff$ is off-diagonal. This implies the existence of a pair of functions $r,s:\Omega\rightarrow \mathbb{R}$ with $r\omega_1=U_\mathfrak{p}$ and $s\omega_2=V_\mathfrak{p}$.
The functions $r$ and $s$ measure the relation between the tangent directions of the surface and the $p$-part of the Maurer--Cartan form of the frame.

The following result characterizes the compatibility between the frame data 
and the harmonicity condition.
\begin{lemma}\label{lemma4.1} The functions $r$ and $s$  satisfy $1+r+s=0$, and the following are equivalent:
\begin{itemize}
    \item $r$ and $s$ are constant,
    \item $\nu$ is harmonic with respect to II.
\end{itemize} 
\end{lemma}
\begin{proof}
    Compute the second derivatives
    \beqas
    (f^{-1}f_x)_y&=&\Ad_F([V_\mathfrak{p}+V_\mathfrak{k},\omega_1]+\partial_y \omega_1)\\
    &=&\Ad_F(s[\omega_2,\omega_1]+[V_\mathfrak{k},\omega_1]+\partial_y \omega_1),
    \eeqas
    and
    \beqas
    (f^{-1}f_y)_x&=&\Ad_F(r[\omega_1,\omega_2]+[U_\mathfrak{k},\omega_2]+\partial_x\omega_2).
    \eeqas
    Then the equation
    \beqa\label{compatibility}
    (f^{-1}f_x)_y-(f^{-1}f_y)_x=[f^{-1}f_x,f^{-1}f_y],
    \eeqa
    turns into
    \beqas
    -s[\omega_1,\omega_2]+[V_\mathfrak{k},\omega_1]+\partial_y\omega_1-r[\omega_1,\omega_2]-[U_\mathfrak{k},\omega_2]-\partial_x\omega_2=[\omega_1,\omega_2],
    \eeqas
    thus
    \beqas
   [V_\mathfrak{k},\omega_1]+\partial_y\omega_1-[U_\mathfrak{k},\omega_2]-\partial_x\omega_2=(1+r+s)[\omega_1,\omega_2].
    \eeqas
    Notice that the rhs is pararel to $e_1$ and the lhs is orthogonal to it. Therefore the previous equation is equivalent to the system
    \beqa\label{mixedpd}
    \left\{\begin{matrix}
         [V_\mathfrak{k},\omega_1]+\partial_y\omega_1-[U_\mathfrak{k},\omega_2]-\partial_x\omega_2&=&0,\\
         (1+r+s)[\omega_1,\omega_2]&=&0.
    \end{matrix}\right.
    \eeqa
    The second equation already gives $1+r+s=0$.

    To prove the doubble implication assume that $\nu$ is harmonic, by Lemma 3.1 this is equivalent to the split Maurer--Cartan system. Assume also there is a point $p\in\Omega$ such that $r(p)\neq 0,-1$. By continuity there is a neighbourhood of $p$ where $r\neq 0,-1$ (if this point does not exist the $r$ must be constant). In this neighbourhood we have that $r\cdot s\neq 0$ so we can multiply the first equation of \eqref{mixedpd} by $rs$ to obtain
    \beqa\label{eq2.2}
    s([V_\mathfrak{k},U_\mathfrak{p}]+r\partial_y\omega_1)-r([U_\mathfrak{k},V_\mathfrak{p}]+s\partial_x\omega_2)=0.
    \eeqa
    Substituting the expressions
    \beqas
    U_\mathfrak{p}=r\omega_1,\quad V_\mathfrak{p}=s\omega_2,
    \eeqas
    into the last two equations of the split Maurer--Cartan system, yields
    \beqas
    \begin{matrix}
    [U_\mathfrak{p},V_\mathfrak{k}]&=&r_y\omega_1+r\partial_y\omega_1,\\
    [U_\mathfrak{k},V_\mathfrak{p}]&=&-s_x\omega_2-s\partial_x\omega_2.
    \end{matrix}
    \eeqas
    And substituting this in \eqref{eq2.2}
    \beqas
    sr_y\omega_1-rs_x\omega_2=0\Rightarrow r_y=s_x=0.
    \eeqas
    The other partial derivative of $r$ is
    \beqas
    r_x=-\partial_x(1+s)=-s_x=0,
    \eeqas
    thus $r$, and therefore $s$, are constant in this neighbourhood of $p$. This shows that the set $r^{-1}(r(p))$ is both open and closed in $\Omega$ so it has to be the whole thing and $r,s$ are constant.

    Assume now that $r$ (and therefore $s=-(r+1)$) is constant. If this constant is either 0 or -1 that implies $\nu_x=0$ or $\nu_y=0$ so $\nu$ is trivially harmonic. Otherwise the first equation of \eqref{mixedpd} is
    \beqas
    \frac{1}{r}([V_\mathfrak{k},U_\mathfrak{p}]+\partial_y U_\mathfrak{p})+\frac{1}{r+1}([U_\mathfrak{k},V_\mathfrak{p}]+\partial_xV_\mathfrak{p})=0.
    \eeqas
    And from the second equation of \eqref{MC-system} we get that
    \beqas
    [U_\mathfrak{k},V_\mathfrak{p}]+\partial_xV_\mathfrak{p}=[V_\mathfrak{k},U_\mathfrak{p}]+\partial_y U_\mathfrak{p}.
    \eeqas
    Substituting in the previous equation yields
    \beqas
    \frac{2r+1}{r(r+1)}([U_\mathfrak{k},V_\mathfrak{p}]+\partial_xV_\mathfrak{p})=0\Rightarrow  [U_\mathfrak{k},V_\mathfrak{p}]+\partial_xV_\mathfrak{p}=0
    \eeqas
    thus $\nu$ is harmonic.
\end{proof}

In particular, this shows that the harmonicity of the Gauss map imposes strong 
restrictions on the geometry of the immersion, linking analytic properties 
to curvature conditions.
\begin{itemize}
    \item $U_\mathfrak{p}$ and $V_\mathfrak{p}$ are linearly independent. This is the same as saying that $\nu$ is an immersion since
    \beqas
    \nu_x=\Ad_F[U_\mathfrak{p},e_1]\quad\text{and}\quad \Ad_F[V_\mathfrak{p},e_1],
    \eeqas
    \item $U_\mathfrak{p}\neq 0$ and $V_\mathfrak{p}=0$ or viceversa. In this case $\nu$ is not an immersion and we have $\nu_x\neq 0$ and $\nu_y=0$.
\end{itemize}
This is a significant difference with the Euclidean case, where if the Gauss map of an immersion is harmonic with respect to the metric given by $\sff$ then it is automatically an immersion. This dichotomy will be fundamental in what follows, leading to two distinct types of geometric behaviour.

\subsubsection{Case 1: $\nu$ is an immersion} The previous discussion shows that in this case we must have $r\neq 0,-1$, therefore
\beqa\label{f1}
f^{-1}f_x=\frac{1}{r}\Ad_FU_\mathfrak{p},\quad f^{-1}f_y=-\frac{1}{r+1}\Ad_FV_\mathfrak{p}, \quad \nu=\Ad_F e_1.
\eeqa
A straightforward computation shows that
\beqas
\fff=\frac{1}{r^2}\langle U_\mathfrak{p},U_\mathfrak{p}\rangle \dd x^2-\frac{2}{r(r+1)}\langle U_\mathfrak{p},V_\mathfrak{p}\rangle \dd x\dd y +\frac{1}{(r+1)^2}\langle V_\mathfrak{p},V_\mathfrak{p}\rangle\dd y^2,
\eeqas
and
\beqas
\sff=-\frac{2r+1}{r(r+1)}\langle [U_\mathfrak{p},V_\mathfrak{p}],\nu\rangle \dd x\dd y.
\eeqas
Therefore the curvature is
\beqas
K+1=\frac{\det II}{\det I}=-(2r+1)^2.
\eeqas

This together with Lemma \ref{lemma4.1} shows:
\begin{theorem}\label{th1}
    Let $\Omega\subset\mathbb{R}^2$ be a domain and
    \beqas
    f:\Omega\longrightarrow \ads
    \eeqas
    a spacelike immersion with negative intrinsiccurvature and such that its Gauss map $\nu$ is also an immersion. Then the following are equivalent:
    \begin{itemize}
        \item The cuvature is constant,
        \item $\nu$ is harmonic with respect to $II$.
    \end{itemize}
\end{theorem}
This result can be seen as a Lorentzian analogue of classical characterizations 
of surfaces via their Gauss map (see, e.g., \cite{ruh1970tension}).

We now prove a converse for the construction \eqref{f1}, showing that every harmonic Gauss map 
gives rise to a corresponding spacelike immersion.
\begin{proposition}
    Let $\Omega\subset \mathbb{R}^{1,1}$ be a simply connected domain,
    \beqas
    \nu:\Omega\longrightarrow\mathbb{H}^2,
    \eeqas
    a harmonic immersion and $F$ an $\SLR$-frame for $\nu$. Consider null coordinates in $(x,y)$ in $\Omega$. Then for every $r\neq 0,-1$ there exists a spacelike immersion (unique up to an initial condition)
    \beqas
    f:\Omega\longrightarrow \ads,
    \eeqas
    of constant curvature
    \beqas
    K+1=-(2r+1)^2,
    \eeqas
    and such that $\nu$ is its Gauss map. In particular $f$ satisfies
    \beqas
    f^{-1}f_x=\frac{1}{r}\Ad_FU_\mathfrak{p},\quad f^{-1}f_y=-      \frac{1}{r+1}\Ad_FV_\mathfrak{p}, \quad \nu=\Ad_F e_1.
    \eeqas
\end{proposition}
\begin{proof}
    Start by considering the auxiliary $\slR$-valued 1-form
    \beqas
    \beta=\frac{1}{4r}[\nu,\nu_x]\dd x-\frac{1}{4(r+1)}[\nu,\nu_y]\dd y.
    \eeqas
    This form satisfies the Maurer-Cartan equation which ensures integrability of the system (see, e.g., \cite{MR900587})
    \beqas
    \dd \beta+\frac{1}{2}[\beta,\beta]&=&\frac{1}{4}(\frac{1}{r}[\nu_y,\nu_x]+\frac{1}{r+1}[\nu_x,\nu_y]+\frac{1}{4r(r+1)}[[\nu,\nu_x], [\nu,\nu_y]])\\
    &=&\frac{1}{4}(\frac{1}{r(r+1)}[\nu_x,\nu_y]+\frac{1}{r(r+1)}[\nu_x,\nu_y])=0.
    \eeqas
    Since we chose $\Omega$ simply connected this means that $\beta$ is integrable, that is, there exists a smooth function
    \beqas
    f:\Omega\longrightarrow\ads,
    \eeqas
    such that $f^{-1}\dd f=\beta$. This implies that
    \beqas
    f^{-1}f_x=\frac{1}{4r}[\nu,\nu_x]=\frac{1}{r}\Ad_FU_\mathfrak{p},\quad f^{-1}f_y=-\frac{1}{4(r+1)}[\nu,\nu_y]=-\frac{1}{r+1}\Ad_FV_\mathfrak{p}.
    \eeqas
    Which immediately shows that $f$ is an immersion, $\nu$ is its Gauss map ($\langle f^{-1}\dd f,\nu\rangle=0$) and
    \beqas
    K+1=-(2r+1)^2.
    \eeqas
\end{proof}
This shows that there is exactly one immersion $f_r$ (up to an initial condition) for each $r\neq 0,-1$ that satisfies the equations $\eqref{f1}$. And these are all the immersion that have $\nu$ as a Gauss map and such that asymptotic coordinates for $f_r$ are null for $\nu$.
In this case, the vectors $[\nu,\nu_x]$ and $[\nu,\nu_y]$ form a basis of the tangent space, which allows us to reconstruct the immersion explicitly. In particular, the immersion condition on $\nu$ ensures that the above construction produces all spacelike immersions having $\nu$ as Gauss map and compatible asymptotic coordinates. This reconstruction formula is a geometric counterpart of the DPW construction 
for harmonic maps, now expressed in terms of surface theory.
\begin{example}
We now apply this construction to the explicit harmonic map obtained in Section 4.
For the harmonic map constructed in Section~4 we have
\[
\nu = \operatorname{Ad}_F e_1, \qquad
F(x,y) =
\begin{pmatrix}
\dfrac{1}{1-xy} & \dfrac{x}{1-xy} \\
\dfrac{y}{1-xy} & \dfrac{1}{1-xy}
\end{pmatrix}.
\]

We first compute the Maurer--Cartan form of $F$. A direct calculation gives
\[
F^{-1} =
\begin{pmatrix}
1 & -x \\
- y & 1
\end{pmatrix},
\]
and
\[
F_x =
\frac{1}{(1-xy)^2}
\begin{pmatrix}
y & 1 \\
y^2 & y
\end{pmatrix},
\qquad
F_y =
\frac{1}{(1-xy)^2}
\begin{pmatrix}
x & x^2 \\
1 & x
\end{pmatrix}.
\]

Thus
\[
F^{-1}F_x =
\frac{1}{1-xy}
\begin{pmatrix}
0 & 1 \\
0 & 0
\end{pmatrix}
= \frac{1}{1-xy} E_+,
\]
\[
F^{-1}F_y =
\frac{1}{1-xy}
\begin{pmatrix}
0 & 0 \\
1 & 0
\end{pmatrix}
= \frac{1}{1-xy} E_-.
\]

Recalling that
\[
E_+ = \tfrac12(e_2 + e_3), \qquad
E_- = \tfrac12(e_2 - e_3),
\]
we obtain
\[
U_{\mathfrak{p}} = \frac{1}{2(1-xy)}(e_2 + e_3), \qquad
V_{\mathfrak{p}} = \frac{1}{2(1-xy)}(e_2 - e_3).
\]

We now compute the commutators appearing in Proposition~5.4. A direct computation yields
\[
[\nu,\nu_x] = 2(1-xy) E_+, \qquad
[\nu,\nu_y] = 2(1-xy) E_-.
\]

Substituting into the formula
\[
f^{-1} df =
\frac{1}{4r}[\nu,\nu_x]\,dx
-
\frac{1}{4(r+1)}[\nu,\nu_y]\,dy,
\]
we obtain
\[
f^{-1} df =
\frac{1}{2r} E_+\, dx
-
\frac{1}{2(r+1)} E_-\, dy.
\]

This system can be integrated explicitly, since $E_+^2 = E_-^2 = 0$. We obtain
\[
f(x,y) =
\exp\left(\frac{x}{2r} E_+\right)
\exp\left(-\frac{y}{2(r+1)} E_-\right).
\]

Using the nilpotency of $E_\pm$, we have
\[
\exp(aE_+) = I + aE_+, \qquad
\exp(bE_-) = I + bE_-,
\]
and therefore
\[
f(x,y)
=
\left(I + \frac{x}{2r} E_+\right)
\left(I - \frac{y}{2(r+1)} E_-\right).
\]

Multiplying, we obtain the explicit expression
\[
f(x,y)
=
\begin{pmatrix}
1 - \dfrac{xy}{4r(r+1)} & \dfrac{x}{2r} \\
-\dfrac{y}{2(r+1)} & 1
\end{pmatrix}.
\]
This provides an explicit family of spacelike surfaces in $H^3_1$ with constant curvature $K$ satisfying
\[
K+1 = -(2r+1)^2.
\]
These surfaces are parametrized by null coordinates, since the coordinate vector fields correspond to lightlike directions. In particular, they are generated by the action of two one-parameter unipotent subgroups of $\mathrm{SL}(2,\mathbb{R})$, reflecting the algebraic structure of the construction.

The parameter $r$ determines the curvature, and varying $r$ produces a one-parameter family of spacelike surfaces with constant curvature $K<-1$.

Finally, we note that the expression degenerates along the set $xy=1$, which corresponds to the boundary of the domain where the construction is defined.
\end{example}

\subsubsection{Case 2: $\nu$ is not an immersion} We now discuss the degenerate case where $\nu$ fails to be an immersion, 
which has no direct analogue in the classical Riemannian setting. In contrast with the previous case, the vectors $[\nu,\nu_x]$ and $[\nu,\nu_y]$ are linearly dependent, and therefore the construction of Proposition~5.4 no longer applies.

However, this does not imply that there are no spacelike immersions having $\nu$ as Gauss map. It only shows that such immersions cannot, in general, be obtained by the previous method.
Without loss of generality we consider only the case $r=-1$. Then
\beqas
f^{-1}f_x=-\Ad_FU_\pp=-\frac{1}{4}[\nu,\nu_x], \quad f^{-1}f_y=[\nu,\omega],
\eeqas
for a certain $\omega:\Omega \rightarrow \slR$. This function must be orthogonal to $\nu$ and not parallel to $\nu_x$. Also the compatibility equation \eqref{compatibility} yields $[\nu,\omega_x]=0$.

\begin{proposition}\label{prop4.4}
    Let $\Omega\subset\mathbb{R}^{1,1}$ be a simply connected domain,
    \beqas
    \nu:\Omega\longrightarrow\mathbb{H}^2,
    \eeqas
    a smooth map and $F$ an $\SLR$-frame for $\nu$. Consider a set of null coordinates $(x,y)$ in $\Omega$ and assume $\nu_x\neq 0$ and $\nu_y=0$. Then for every smooth function $\omega:\Omega\rightarrow \slR$ satisfying $\langle\omega,\nu\rangle =0$, $[\nu_x,\omega]\neq 0$ and $\omega_x=0$ there exists a spacelike immersion (unique up to initial condition)
    \beqas
    f:\Omega\longrightarrow\ads,
    \eeqas
    of constant curvature $-1$ and such that $\nu$ is its Gauss map. In particular we have
    \beqas
    f^{-1}f_x=-\Ad_FU_\pp, \quad f^{-1}f_y=[\nu,\omega],\quad \nu=\Ad_Fe_1.
    \eeqas
\end{proposition}
\begin{proof}
    Consider the auxiliary $\slR$-valued 1-form
    \beqas
    \beta=-\frac{1}{4}[\nu,\nu_x]\dd x+[\nu,\omega]\dd y.
    \eeqas
    It satisfies the Maurer-Cartan equation
    \beqas
    \dd\beta+\frac{1}{2}[\beta,\beta]=[\nu,\omega_x]=0.
    \eeqas
    Thus we can integrate $f^{-1}\dd f=\beta$ to obtain a smooth map $f:\Omega\rightarrow \ads$ satisfying
    \beqas
    f^{-1}f_x=-\Ad_FU_\pp, \quad f^{-1}f_y=[\nu,\omega].
    \eeqas
    The condition $[\nu_x,\omega]\neq 0$ implies that $f$ is an immersion, and $\langle \omega,\nu\rangle=0$ implies $\langle \nu,\dd f\rangle=0$. That is, $\nu$ is the Gauss map for the immersion $f$.

    For the claim on the curvature, compute
    \beqas
    \fff=\frac{1}{4}\langle \nu_x,\nu_x\rangle\dd x^2-2\langle \nu_x,\omega\rangle\dd x\dd y+4\langle \omega,\omega\rangle \dd y^2,
    \eeqas
    and
    \beqas
    \sff=\langle[\nu_x,\omega],\nu\rangle \dd x\dd y.
    \eeqas
    Thus
    \beqas
    K+1=\frac{\det \sff}{\det \fff}=-1.
    \eeqas
\end{proof}
This shows that even when $\nu$ is not an immersion, one can still construct spacelike immersions with prescribed Gauss map, although the structure of the construction differs significantly from the nondegenerate case.
\begin{example}
    Consider the map
    \beqas
    \nu(x,y)=\begin{pmatrix}
        -\sinh x&-\cosh x\\ \cosh x&\sinh x
    \end{pmatrix} :\mathbb{R}^{1,1}\longrightarrow \mathbb{H}^2.
    \eeqas
    We can define an $\SLR$-frame for $\nu$ as
    \beqas
    F=\begin{pmatrix}
        \frac{1}{\sqrt{\cosh x}}&-\frac{\sinh x}{\sqrt{\cosh x}}\\
        0&\sqrt{\cosh x}
    \end{pmatrix}:\mathbb{R}^{1,1}\longrightarrow \SLR.
    \eeqas
    Its Maurer-Cartan form is
    \beqas
    \alpha=F^{-1}\dd F=\frac{1}{2}(e_1\sech x-e_2\sech x+e_3\tanh x)\dd x.
    \eeqas
    Write $\omega=\Ad_F(\omega^1e_1+\omega_2e^2+\omega^3e_3)$. The condition $\langle \omega, \nu\rangle=0$ implies that $\omega_1=0$. The second condition $[\nu,\omega_x]=0$ can be written as the system
    \beqas
    \left\{\begin{matrix}
        \omega_x^2&=&\sech x\cdot \omega^3,\\
        \omega_x^3&=&-\sech x\cdot \omega^2,
    \end{matrix}\right. 
    \eeqas
    which general solution is given by
    \beqas
    \omega^2=-A\tanh x-B\sech x,\quad \omega^3=-A\sech x+B\tanh x.
    \eeqas
    For the last condition we get
    \beqas
    0\neq[\nu_x,\omega]=-2B(\cosh x\cdot e_1+\sinh x\cdot e_3)\Rightarrow B\neq 0.
    \eeqas
    Thus the 1-form $\beta$ from Proposition \ref{prop4.4} is
    \beqas\label{eq4.4}
    \beta=\frac{1}{2}e_2\dd x-2(B\sinh x\cdot e_1-A\cdot e_2+B\cosh x\cdot e_3)\dd y.
    \eeqas
    In order to get a explicit example let $\theta\in(0,\pi)$ and set
    \beqas
    A=\frac{1}{4}\cos(2\theta), \quad B=\frac{1}{4}\sin(2\theta).
    \eeqas
    Then we can integrate the equation $f^{-1}\dd f=\beta$ with initial condition $f(0,0)=\cos\theta\cdot e_0+\sin\theta\cdot e_1$ to obtain 
    \beqas
    f(x,y)=(\cos\theta\cosh\frac{x+y}{2}, \sin\theta\cosh\frac{x-y}{2}, \cos\theta\sinh\frac{x+y}{2}, \sin\theta\sinh\frac{x-y}{2}).
    \eeqas
\end{example}

Summarizing, the correspondence between Lorentz harmonic maps and spacelike immersions in $H^3_1$ exhibits two distinct regimes, depending on whether the Gauss map is an immersion. In the nondegenerate case, the immersion can be reconstructed explicitly via the frame and depends on a single parameter. In the degenerate case, additional data is required, and the construction becomes less rigid.

\section{Geometric Cauchy problem}
We now formulate and solve the geometric Cauchy problem for spacelike surfaces in $\ads$, 
using the correspondence developed in the previous sections.
Let $I\subset\mathbb{R}$ be an interval. The Cauchy data that we need consists of
\beqas
    \Tilde{f}:I\longrightarrow \ads,\quad \Tilde{\nu}:I\longrightarrow \mathbb{H}^2,\quad 0<\rho \neq 1,
\eeqas
where $\Tilde{f}$ is a spacelike curve and $\langle \dd\Tilde{f},\Tilde{\nu}\rangle =0$. We want to find a spacelike surface $f$ of constant extrinsic curvature $-\rho^2$ that contains the curve $\Tilde{f}$ and such that its Gauss map $\nu$ agrees with $\Tilde{\nu}$ along $\Tilde{f}$.

The key idea is to express the geometric conditions in terms of the frame and reduce the problem to the construction of a harmonic map with suitable initial data. Set $(t,s)=(x+y,x-y)/2$ and assume there exists a spacelike immersion $f:I\times I\rightarrow \ads$ with Gauss map $\nu$ such that
\beqas
f(t,0)=\Tilde{f}(t),\quad \nu(t)=\Tilde{\nu}(t,0).
\eeqas
Notice that $f$ solves the geometric Cauchy problem. 

Recall the equations \eqref{f1}
\[
f^{-1}f_x = \frac{\alpha}{2}[\nu,\nu_x],\quad
f^{-1}f_y = \frac{\beta}{2}[\nu,\nu_y],
\quad2\alpha\beta+\alpha+\beta=0.
\]
Where
\beqas
\alpha=\frac{1}{\rho-1}\quad\text{and}\quad \beta=\frac{-1}{\rho+1}.
\eeqas
Then because it is $\partial_t=\partial_x+\partial_y$ we get
\beqa
\left\{\begin{matrix}
f^{-1}f_t&=&\frac{1}{2}[\nu(t),\alpha\nu_x(t)+\beta\nu_y(t)],\\
\nu'(t)&=&\nu_x(t)+\nu_y(t).   
\end{matrix}\right.
\eeqa
Set
\beqas
w(t)=-\frac{1}{2}[\tilde{\nu}(t),\tilde{f}^{-1}\tilde{f}_t]=\alpha \nu_x(t)+\beta\nu_y(t),
\eeqas
The quantity $w(t)$ can be interpreted as a measure of the rotation of the normal along the initial curve. Then
\beqas
\tilde{\nu}'(t)-\frac{w}{\beta}=(1-\frac{\alpha}{\beta})\nu_x(t)\Rightarrow \tilde{\nu}'(t)+(\rho+1)w=\frac{2\rho}{\rho-1}\nu_x(t),
\eeqas
thus
\beqas
\nu_x(t)=\frac{\rho-1}{2\rho}(\tilde{\nu}'(t)+(\rho+1)w).
\eeqas
We can now state the solution to the geometric Cauchy problem.

\begin{corollary}\label{coro1}
    Let $I\subset \mathbb{R}$ be an open interval. Given Cauchy data
    \beqas
    \Tilde{f}:I\longrightarrow \ads,\quad \Tilde{\nu}:I\longrightarrow \mathbb{H}^2,\quad 0<\rho\neq 1,
    \eeqas
    with $\Tilde{f}$ spacelike and
    \beqas
    \langle \Tilde{f}^{-1}\Tilde{f}_t,\Tilde{\nu}\rangle =0,\quad
    \langle \Tilde{\nu}'+(\rho+1)w,\Tilde{\nu}'-(\rho-1))w\rangle\not\equiv 0.
    \eeqas
    This nondegeneracy condition ensures that the resulting harmonic map is an immersion along the initial curve.
    Then there exists a spacelike immersion
    \beqas
    f:I\times I\longrightarrow \ads
    \eeqas
    of constant extrinsic curvature
    \beqas
    K+1=-\rho^2,
    \eeqas
    and such that along the diagonal $x=y$ we have
    \beqas
    f(x,x)=\Tilde{f}(x),\quad \nu(x,x)=\Tilde{\nu}(x),\quad t\in I.
    \eeqas
    
\end{corollary}
This provides a constructive solution to the geometric Cauchy problem, 
extending classical results in integrable surface theory to the Lorentzian setting.
\begin{proof}
Set as before
\beqas
w(t)=-\frac{1}{2}[\Tilde{\nu}(t),\Tilde{f}^{-1}\Tilde{f}_t].
\eeqas
We now translate the geometric Cauchy data into the analytic data required in Theorem~4.1.
\beqas
N_0(t)=\Tilde{\nu}(t),\quad N_1(t)=\frac{\rho -1}{2\rho}(\Tilde{\nu}'(t)+(\rho +1)w(t)).
\eeqas
Then
\beqas
\langle N_1(t),N_0(t)\rangle=\frac{\rho-1}{2\rho}\langle \Tilde{\nu}'(t),\Tilde{\nu}\rangle+\frac{\rho^2-1}{2\rho}\langle w(t),\Tilde{\nu}(t)\rangle=0,
\eeqas
and
\beqas
\langle N_1(t),2N'_0(t)-N_1(t)\rangle =\frac{\rho^2-1}{4\rho^2}\langle \Tilde{\nu}'+(\rho+1)w,\nu'-(\rho-1)w\rangle \not\equiv 0.
\eeqas
Thus we are in the hypothesis of Theorem 4.1. There exists a harmonic map
\beqas
\nu:I\times I\longrightarrow \mathbb{H}^2,
\eeqas
such that
\beqas
\nu(x,x)=N_0(x)=\Tilde{\nu}(x),\quad \nu_x(x,x)=N_1(x)=\frac{\rho -1}{2\rho}(\Tilde{\nu}'(x)+(\rho +1)w(x)).
\eeqas
In particular this means that in coordinates $(t,s)$ in the diagonal $s=0$ we have that
\beqas\label{expw}
\frac{1}{2}\nu_t(t,0)=\nu_x(x,x)=\frac{\rho -1}{2\rho}(\nu_t(t,0)+(\rho +1)w(t))\Rightarrow w(t)=\frac{\nu_t(t,0)}{\rho^2-1}.
\eeqas

Now to obtain the immersion $f$ we follow the construction of Section 5.2. We can integrate the equation
\beqas\label{eqlast}
f^{-1}\dd f=\alpha=\frac{[\nu,\nu_x]}{2(\rho-1)}\dd x-\frac{[\nu,\nu_y]}{2(\rho+1)}\dd y,
\eeqas
with initial condition $f(x_0,x_0)=\Tilde{f}(x_0)$, to obtain a spacelike immersion
\beqas
f:I\times I\longrightarrow \ads,
\eeqas
of constant Gauss curvature
\beqas
K+1=-\left(\frac{\frac{1}{\rho-1}+1}{\frac{1}{\rho-1}}\right)^2=-\rho^2.
\eeqas

To see that it contains the initial curve $\Tilde{f}$ write $\alpha$ in the coordinates $(t,s)$
\beqas
\alpha(t,s)=\frac{[\nu,\nu_t+\nu_s]}{4(\rho-1)}(\dd t+\dd s)-\frac{[\nu,\nu_t-\nu_s]}{4(\rho+1)}(\dd t-\dd s),
\eeqas
hence along the diagonal $s=0$ we have
\beqas
\alpha(t,0)&=&\frac{[\nu,\nu_t]}{4(\rho-1)}\dd t-\frac{[\nu,\nu_t]}{4(\rho+1)}\dd t=\frac{1}{2}\frac{[\nu,\nu_t]}{\rho^2-1}\dd t\\&=&\frac{1}{2}[\nu, w]\dd t=\Tilde{f}^{-1}\Tilde{f}_t\dd t=\Tilde{f}\dd\Tilde{f}.
\eeqas
Both $f$ and $\Tilde{f}$ integrate $\alpha$ along the diagonal and we chose $f$ with initial condition $f(x_0,y_0)=\Tilde{f}(x_0)$. Thus, by uniqueness of solution, it must be $f=\Tilde{f}$ along the diagonal and the proof is complete.

This completes the construction and shows that the geometric Cauchy problem admits a local solution under the stated assumptions.

\end{proof}
This result establishes a direct bridge between the analytic theory of Lorentz harmonic maps 
and the geometry of spacelike surfaces in anti-de Sitter space, providing a unified framework 
for the construction of surfaces from prescribed boundary data.
\begin{example}
We illustrate the geometric Cauchy problem by recovering explicitly the surface constructed in Sections~4 and~5.

We prescribe the initial data along a curve by setting
\[
\tilde f(t)
=
\begin{pmatrix}
1 - \dfrac{t^2}{4r(r+1)} & \dfrac{t}{2r} \\
-\dfrac{t}{2(r+1)} & 1
\end{pmatrix},
\qquad
\tilde \nu(t)
=
\begin{pmatrix}
-\dfrac{2t}{1-t^2} & -\dfrac{1+t^2}{1-t^2} \\
\dfrac{1+t^2}{1-t^2} & \dfrac{2t}{1-t^2}
\end{pmatrix}.
\]

A direct computation shows that $\tilde f(t) \in \mathrm{SL}(2,\mathbb{R})$ and $\tilde \nu(t) \in H^2$. Moreover, one verifies that $\tilde \nu(t)$ is orthogonal to $\tilde f^{-1}\tilde f_t$, so the compatibility condition is satisfied.

\medskip

We compute
\[
\tilde f^{-1}\tilde f_t =
\begin{pmatrix}
0 & \dfrac{1}{2r} \\
-\dfrac{1}{2(r+1)} & 0
\end{pmatrix},
\qquad
w(t) = -\frac12[\tilde \nu, \tilde f^{-1}\tilde f_t].
\]

Using the formulas of Corollary~6.1, we obtain the Cauchy data
\[
N_0(t) = \tilde \nu(t), \qquad N_1(t),
\]
which coincide with those computed in Section~4. In particular, the nondegeneracy condition is satisfied.

\medskip

By Theorem~4.1, the corresponding harmonic map is therefore given explicitly by
\[
\nu(x,y)
=
\begin{pmatrix}
-\dfrac{x+y}{1-xy} & -\dfrac{1+xy}{1-xy} \\
\dfrac{1+xy}{1-xy} & \dfrac{x+y}{1-xy}
\end{pmatrix}.
\]

\medskip

Applying Proposition~5.4, we obtain a one-parameter family of spacelike immersions $f$ with Gauss map $\nu$, depending on the parameter $r$.

To recover the surface solving the geometric Cauchy problem, we fix the value of $r$ corresponding to the prescribed initial data. This yields explicitly
\[
f(x,y)
=
\begin{pmatrix}
1 - \dfrac{xy}{4r(r+1)} & \dfrac{x}{2r} \\
-\dfrac{y}{2(r+1)} & 1
\end{pmatrix},
\]
which satisfies $f(t,t)=\tilde f(t)$.

\medskip

Finally, this surface has constant curvature
\[
K+1 = -(2r+1)^2,
\]
and provides the unique solution to the geometric Cauchy problem with the given initial data.

To visualize the solution, we represent the immersion $f(x,y)$ by projecting its matrix entries onto $\mathbb{R}^3$. This provides an affine model of the surface suitable for graphical representation.
\begin{figure}[H]
\centering

\begin{subfigure}{0.48\textwidth}
    \centering
    \includegraphics[width=\linewidth]{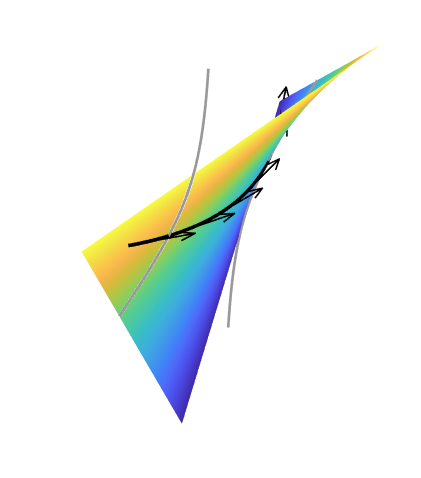}
\end{subfigure}
\hfill
\begin{subfigure}{0.48\textwidth}
    \centering
    \includegraphics[width=\linewidth]{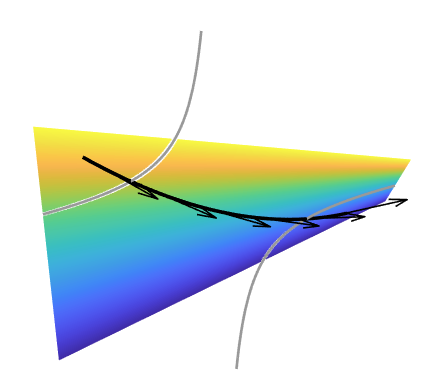}
\end{subfigure}

\caption{
Solution of the geometric Cauchy problem for $r=2$. 
The black curve represents the initial data $\tilde f(t)$, 
the grey curves correspond to the singular set $xy=1$, 
and the arrows indicate the direction of the vector field $w(t)$ along the curve. 
Two different viewpoints are shown to illustrate the geometry of the surface.
}
\end{figure}

\begin{figure}[H]
\centering

\begin{subfigure}{0.48\textwidth}
    \centering
    \includegraphics[width=\linewidth]{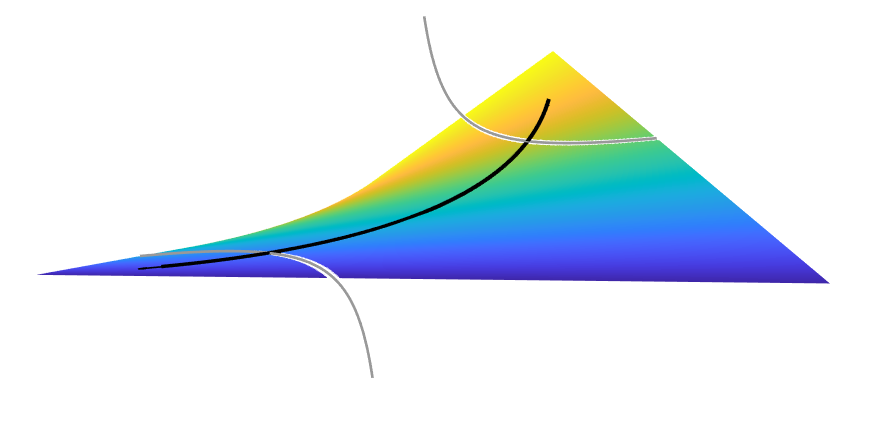}
\end{subfigure}
\hfill
\begin{subfigure}{0.48\textwidth}
    \centering
    \includegraphics[width=\linewidth]{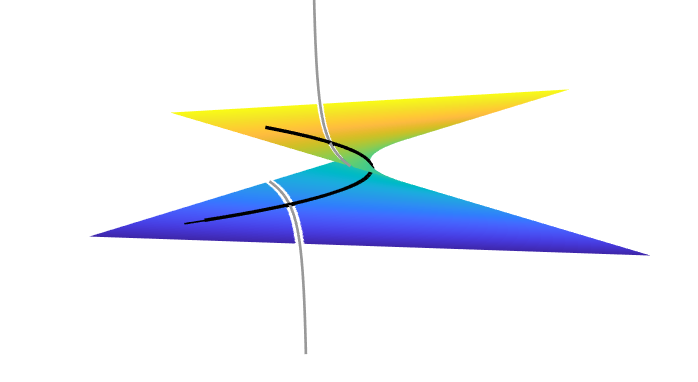}
\end{subfigure}

\caption{
Solution of the geometric Cauchy problem for $r=0.3$ from two different viewpoints.
}
\end{figure}
\end{example}

%% file: parallel_surfaces.tex
\section{Parallel surfaces}

The construction of parallel surfaces is a classical tool in the differential 
geometry of surfaces in space forms. It provides a natural way to generate 
new immersions from a given one, while preserving fundamental geometric data 
such as the Gauss map. More generally, parallel surface constructions can be 
understood as geometric transformations which preserve curvature properties 
and are closely related to the theory of integrable systems and nonlinear 
partial differential equations (see, e.g., \cite{tenenblat1998transformations}).

In particular, it is a classical fact that surfaces satisfying curvature 
conditions, such as constant Gaussian curvature or constant mean curvature, 
often appear in one-parameter families of parallel surfaces, relating different 
geometric classes. These transformations play a central role in the interaction 
between differential geometry and integrable PDEs, where they correspond to 
transformations between solutions of the associated equations.

In our framework, this perspective naturally appears: the harmonic Gauss map 
constructed via the DPW method gives rise not to a single surface, but to a 
whole family of associated surfaces. The parallel surface construction 
provides a geometric realization of this phenomenon, relating constant Gauss 
curvature and constant mean curvature geometries in anti-de Sitter space.

\medskip

The main goal of this section is to show that this correspondence can be used 
to transfer the geometric Cauchy problem between these two classes of surfaces. 
More precisely, we will show that the geometric Cauchy problem for constant mean 
curvature surfaces can be reduced to the corresponding problem for constant 
Gauss curvature surfaces, allowing us to apply the machinery developed in the 
previous sections.

\medskip

In this section we construct CMC spacelike surfaces starting from a CGC surface.
Consider a spacelike immersion 
\[
f:\Omega\longrightarrow \ads\subset \mathbb{R}^4_2
\]
and let $N$ be its normal. We define the parallel surface to $f$ at a distance 
$\theta\in\mathbb{R}$ by
\beqas
f^\theta=f\cos \theta+N\sin\theta.
\eeqas
This map is not necessarily an immersion. Let $S$ be the shape operator of the immersion $f$, we have
\beqa\label{derivatives3}
\begin{pmatrix}
    f^\theta_x\\f^\theta_y
\end{pmatrix}=(I\cos\theta+S\sin\theta)\begin{pmatrix}
    f_x\\f_y
\end{pmatrix}.
\eeqa
Thus $f^\theta$ is an immersion if and only if the matrix $I\cos\theta +S\sin\theta$ is non-singular, that is
\beqas
\sin^2\theta\det S+\cos\theta\sin\theta \trace S+\cos^2\theta\neq 0.
\eeqas
We can write this in terms of the intrinsic curvature $K$ and the mean curvature $H$ of the surface 
$f$ as
\beqas
K\sin^2\theta-H\sin(2\theta)+1\neq 0.
\eeqas

In the following we assume that $f^\theta$ is an immersion. The vector 
\[
N^\theta=N\cos\theta-f\sin\theta
\]
is a normal for the surface $f^\theta$. While the normal changes, the Gauss map 
remains the same along all parallel surfaces.

The shape operator $S_\theta$ associated to $f_\theta$ is
\beqa\label{shape}
S_\theta=\frac{\overline{S}\sin^2\theta-IK\sin\theta\cos\theta-S\cos^2\theta}{K\sin^2\theta-H\sin(2\theta)+1}.
\eeqa
Taking determinants in equation \eqref{shape}
\beqas
K^\theta=\frac{K\cos(2\theta)+2H\sin(2\theta)}{K\sin^2\theta-H\sin(2\theta)+1}.
\eeqas
Therefore, if $f$ is a CMC immersion and we choose $\theta$ such that
\beqas
\tan(2\theta)=\frac{1}{H},
\eeqas
then $f^\theta$ is an immersion of constant curvature 
\[
K^\theta=\frac{1}{\tan^2\theta}-1.
\]
On the other hand, taking traces in \eqref{shape}
\beqas
H^\theta=\frac{K\sin\theta\cos\theta-H\cos(2\theta)}{K\sin^2\theta-H\sin(2\theta)+1}.
\eeqas
Therefore, if $f$ has constant curvature $K$ and we choose $\theta$ such that
\beqas
\tan^2\theta=\frac{1}{K+1},
\eeqas
then $f^\theta$ is a CMC immersion with
\beqas
H^\theta=\frac{1}{\tan(2\theta)}.
\eeqas

All in all, we obtain:

\begin{theorem}\label{paralelsl}
    Let 
    \[
    f:\Omega\longrightarrow\ads
    \]
    be a spacelike immersion with mean curvature $H$ and extrinsic curvature $K$. Then:
    \begin{enumerate}
        \item If $f$ has constant mean curvature $H$, for $\theta$ such that
        \beqas
        \tan(2\theta)=\frac{1}{H}\quad\text{and}\quad K\sin^2\theta-H\sin(2\theta)+1\neq 0,
        \eeqas
        then $f^\theta$ is CGC with 
        \[
        K^\theta=\frac{1}{\tan^2\theta}-1.
        \]
        \item If $f$ is CGC, for $\theta$ such that
        \beqas
        \tan^2\theta =\frac{1}{K+1}\quad\text{and}\quad K\sin^2\theta-H\sin(2\theta)+1\neq 0,
        \eeqas
        then $f^\theta$ is CMC with
        \beqas
        H^\theta=\frac{1}{\tan(2\theta)}.
        \eeqas
    \end{enumerate}
\end{theorem}

\medskip

The previous theorem shows that constant mean curvature and constant Gauss 
curvature surfaces are locally equivalent, up to the parallel surface construction. 
This allows us to transfer constructions and existence results between both 
classes.

In particular, we obtain the following reformulation of the geometric Cauchy 
problem.

\begin{proposition}
Let $f:\Omega\to \ads$ be a spacelike immersion of constant mean curvature 
$H\neq 0$, and let $(f,N)$ be its geometric Cauchy data along a curve $\gamma$. 
Let $\theta$ satisfy
\[
\tan(2\theta)=\frac{1}{H}.
\]
Assume that
\[
K\sin^2\theta - H\sin(2\theta)+1 \neq 0.
\]
Then the parallel surface
\[
\tilde f = f\cos\theta + N\sin\theta
\]
is a spacelike immersion of constant Gauss curvature. Moreover, its geometric 
Cauchy data along $\gamma$ are given by
\[
\tilde f|_\gamma = f|_\gamma \cos\theta + N|_\gamma \sin\theta,
\]
\[
\tilde N|_\gamma = N|_\gamma \cos\theta - f|_\gamma \sin\theta.
\]
\end{proposition}

\medskip

Therefore, locally and away from singularities of the parallel transformation, 
the geometric Cauchy problem for constant mean curvature surfaces can be reduced 
to the corresponding problem for constant Gauss curvature surfaces. In particular, 
the existence results obtained in the previous sections for CGC surfaces can be 
applied to produce CMC surfaces via the parallel surface construction.